\documentclass[12pt,reqno]{amsart}

\usepackage{amssymb, amsmath, amsfonts, latexsym}
\usepackage{enumerate}
\usepackage{color}
\newtheorem{thm}{Theorem}[]
\newtheorem{lem}{Lemma}[section]

\newtheorem{prop}{Proposition}[]

\newtheorem{rmk}{Remark}[section]
\theoremstyle{definition}

\numberwithin{equation}{section} \theoremstyle{remark}

\title[Fisher information]{\bf Fisher information approximation of random orthogonal matrices by Gaussian matrices} 

\author{Y{utong} Chen}

\author{Y{utao} Ma}
\address{School of Mathematical Sciences $\&$ Laboratory  of Mathematics and Complex Systems of Ministry of Education, Beijing Normal University, 100875 Beijing, China.} 
\thanks{The research of Yutao Ma was supported in part by NSFC 12171038 and 985 Projects.}
\email{mayt@bnu.edu.cn}

\author{S{huhong} Xie} 
\author{Zhuoya Yao}

\begin{document}
\maketitle

\begin{abstract}
Let ${\Gamma}_n$ be an $n\times n$ Haar-invariant orthogonal matrix.  Let ${ Z}_n$ be the $p\times q$
upper-left submatrix of
${\Gamma}_n$ and 
${G}_n$ be a $p\times q$ matrix whose $pq$ entries are independent standard normals, where $p$ and $q$ are two positive integers. Let $\mathcal{L}(\sqrt{n} {Z}_n)$ and $\mathcal{L}({G}_n)$ be their joint distribution, respectively. 
Consider the Fisher information $I(\mathcal{L}(\sqrt{n} { Z}_n)|\mathcal{L}(G_n))$ between the distributions of $\sqrt{n} {Z}_n$ and ${ G}_n.$ In this paper, we conclude that  $$I(\mathcal{L}(\sqrt{n} {Z}_n)|\mathcal{L}(G_n))\longrightarrow  0 $$ as $n\to\infty$ if $pq=o(n)$ and it does not tend to zero if $c=\lim\limits_{n\to\infty}\frac{pq}{n}\in(0, +\infty).$ Precisely, we obtain that 
$$I(\mathcal{L}(\sqrt{n} {Z}_n)|\mathcal{L}(G_n))=\frac{p^2q(q+1)}{4n^2}(1+o(1))$$
when $p=o(n).$ 
 \end{abstract}



\section{Introduction}\label{chap:intro}

Let $\Gamma_n=(\gamma_{ij})_{n\times n}$ be an Haar-invariant orthogonal matrix and let $Z_n$ be the $p\times q$
top-left submatrix of 
$\Gamma_n,$ where $p=p_n$ and $q=q_n$ are two positive integers.
It has long been observed that the entries of $\sqrt{n}{\Gamma}_n$ could be regarded as jointly Gaussians.  Historically, authors show that the distance between $\sqrt{n}{Z}_n$ and $G_n$, say, $d(\sqrt{n}{Z}_n, {G}_n)$ goes to zero under condition $(p, q)=(1,1)$, $(p, q)=(\sqrt{n},1)$, $(p, q)=(o(n),1), $ or $(p, q)=(n^{1/3}, n^{1/3}).$ Readers are referred to, for instance, Maxwell \cite{max75, max78}, Poincar\'{e} \cite{poicare}, Stam \cite{stam}, Diaconis {\it et al}. \cite{DLE} and Collins \cite{Collins}. A more detailed recounts can be seen from Diaconis and Freedman \cite{DF87} and Jiang \cite{Jiang06}.

Obviously, with more research being done, it is known that the values of $p$ and $q$ become larger and
larger such that $d(\sqrt{n}{Z}_n, {G}_n)$ goes to zero.
Diaconis \cite{persi03} then asks the largest values of $p$ and $q$ such that the distance
between $\sqrt{n}Z_n$ and $G_n$ goes to zero. Jiang \cite{Jiang06}  settles the problem by showing that
$p=o(n^{1/2})$ and $q=o(n^{1/2})$ are the largest orders to make the total variation distance go to
zero. If the distance is the maximum norm,  Jiang \cite{Jiang06} further proves  that the largest order of
$q$ is $\frac{n}{\log n}$ with $p=n$. 

Then, in \cite{JM19} the second author with Jiang prove that for a non-square submatrix, all the total variation distance, the Kullback-Leibler divergence and the Hellinger distance  go to zero if $pq=o(n)$ and do not converge to zero if $pq=O(n)$ and the Euclidean distance goes to zero if $pq^2=o(n).$   Simultaneously, Stewart \cite{St20} verifies the approximation by total variation distance as long as $pq=o(n)$ by a different approach. Erd\H{o}s and Mckenna \cite{EM24} continued to show that when the extremal statistics of the corresponding quadratic forms are concerned, $\sqrt{n} Z_n$ are still jointly Gaussian even when $p=O(n^{1/2-\delta})$ and $q=n.$ Hereafter, the notation $t_n=O(n)$ means $c_1 n<t_n<c_2 n$ for two constants $0<c_1< c_2.$

We review the proof in \cite{JM19}, for both Kullback-Leibler divergence or the total variation distance, a very useful tool is the Pinsker inequality for multi-dimensional Gaussian, which helps to connect a bridge between the total variation distance and the Kullback-Leibler divergence. They also list the question on other distances, for example, the Fisher information and Wasserstein distance. It is well known that the joint distribution of ${\bf G}_n$ satisfies the logarithmic Sobolev inequality, which guarantees  
$$2 D_{\rm KL}(\mathcal{L}(\sqrt{n}{Z}_n)|\mathcal{L}({G}_n))\le  I(\mathcal{L}(\sqrt{n}{ Z}_n)|\mathcal{L}({G}_n)) $$
for any $n\ge 1.$ Here, $D_{\rm KL}$ and $I$ are Kullback-Leibler divergence and the Fisher information, respectively. The approximation of random orthogonal matrices by standard Gaussians  in \cite{JM19} with respect to the Kullback-Leibler divergence,  together with the logarithmic Sobolev inequality, guarantees that the approximation under Fisher information fails when $pq=O(n).$ Therefore, we wonder whether the condition $pq=o(n)$ is sufficient to the Fisher information approximation , which is the primary inspiration for us. Our answer is positive. Before stating the main result, we recall the Kullback-Leibler divergence and the Fisher information, which are defined as   
$$\aligned D_{\rm KL}(\mu|\nu)&=\int_{\mathbb{R}^m} f(x)\log \frac{f(x)}{g(x)}  dx;\\
I(\mu|\nu)&=\int_{\mathbb{R}^m}\|\nabla \sqrt{f/g}\|^2 g(x)dx
\endaligned$$ 
for any probabilities $\mu$ and $\nu$ on $(\mathbb{R}^m, \|\cdot\|)$ provided $\mu$ and $\nu$ having density functions $f$ and $g$ with respect to $dx$ the Lebesgue measure, respectively and here $\|\cdot\|$ is the classical Euclidean distance on $\mathbb{R}^m.$ 

\begin{thm}\label{main1}
Given $p$ and $q$ two parameters satisfying $1\le q\le p\le n.$ For each $n\geq 1$, let ${Z}_n$ and ${G}_n$ be the $p\times q$ submatrices aforementioned. Then 
$$I(\mathcal{L}(\sqrt{n} {Z}_n)|\mathcal{L}(G_n))\longrightarrow  0 $$ as $n\to\infty$ if $pq=o(n)$ and it does not tend to zero if $pq=O(n).$ Precisely, we obtain that 
\begin{equation}\label{keyres} 
I(\mathcal{L}(\sqrt{n} {Z}_n)|\mathcal{L}(G_n))=\frac{p^2q(q+1)}{4n^2}(1+o(1))\end{equation}
when $p=o(n).$  
\end{thm}

As mentioned above, the joint distribution of $G_n$ satisfies the logarithmic Sobolev inequality, which guarantees that $I(\mathcal{L}(\sqrt{n} {Z}_n)|\mathcal{L}(G_n))$ does not tend to zero when $pq=O(n).$ We only need to prove the expression 
\eqref{keyres}.

Let $(\lambda_k)_{1\le k\le q}$ be the eigenvalues of $Z'_n Z_n.$  The proof of \eqref{keyres} consists of four main steps 

\begin{enumerate}
\item[({\bf 1}).] Express   
$I(\mathcal{L}(\sqrt{n}{Z}_n)|\mathcal{L}({G}_n)$  as a functional of expectations relating to $(\lambda_k)_{1\le k\le q},$ which is the most complicated part.  
\item[({\bf 2}).] Joint density function of  $(\lambda_k)_{1\le k\le q},$ which enables us to identify $(\lambda_1, \cdots, \lambda_q)$ as a particular $1$-Jacobi ensemble.      
\item[ ({\bf 3}).] The corresponding expectations obtained in the first step. This part involves the tridiagonal random matrix characterization of  beta-Jacobi ensemble (see \cite{ES08}) and it takes a complicated steps.
\item[({\bf 4}).] The final statement of the proof. 	
\end{enumerate}

\section{Proof of Theorem \ref{main1}}

In this section, we are going to organize the proof as planned in the introduction. 

First we give the expression of  $I(\mathcal{L}(\sqrt{n}{Z}_n)|\mathcal{L}({G}_n)$ via the eigenvalues of $Z_n'Z_n.$

\subsection{The expression of Express   
$I(\mathcal{L}(\sqrt{n}{Z}_n)|\mathcal{L}({G}_n)$ }

\quad \quad\quad\quad\quad\quad\quad\quad\quad\quad\quad\quad\quad\quad\quad\quad\quad
\\

We first borrow
Proposition 2.1 by Diaconis, Eaton and Lauritzen \cite{DLE} or Proposition 7.3 by Eaton \cite{ME2}, which offer the crucial joint density function of entries of $\mathbf Z_n.$ 

\begin{lem}\label{del} Let $\Gamma_n$ be an $n\times n$ random matrix which is uniformly
distributed on the orthogonal group $O(n)$ and let $Z_n$ be the upper-left $ p\times q$  submatrix   of $\Gamma_n.$ If $p+q\leq n$ and $q\leq p,$ then the joint density function of entries of $Z_n$ is
$$
f(z)=(\sqrt{2\pi})^{-pq}\frac{\omega(n-p, q)}{\omega(n, q)}\left\{\rm{det}(\mathbf I_{q}-z'z)^{(n-p-q-1)/2}\right\}I_0(z'z)
$$
where $\mathbf I_{q}$ is the $q\times q$ identity matrix, ${\rm det}(A)$ means the determinant of the matrix $A,$  $I_0(z'z)$ is the indicator function of the set that all $q$ eigenvalues of $z'z$ are in $(0,1),$  and $\omega(\cdot, \cdot)$ is the Wishart constant defined by
$$
\frac{1}{\omega(s, t)}=\pi^{t(t-1)/4}2^{st/2}\prod_{j=1}^t\Gamma\left(\frac{s-j+1}{2}\right).
$$
Here $t$ is a positive integer and $s$ is a real number, $s>t-1.$ When $p< q,$ the density of $Z_n$ is obtained by interchanging $p$ and $q$ in the above Wishart constant.
\end{lem}

Next, we express the Fisher information $I(\mathcal{L}(\sqrt{n}{Z}_n)|\mathcal{L}({G}_n))$ as a functional of the eigenvalues of ${Z}_n' {Z}_n.$

\begin{lem}\label{expressofI} For each $n \geq 1$, let ${Z}_n$ and ${G}_n$ be the $p \times q$ submatrices aforementioned.
Let $\lambda_1, \dots, \lambda_q$ be the $q$ eigenvalues of $ {Z}_n' {Z}_n.$  Set $c_n=\frac{n-p-q-1}{2}.$ Then
$$
    I(\mathcal{L}(\sqrt{n}{Z}_n)|\mathcal{L}({G}_n)) =\frac{n}4\sum_{k=1}^q\mathbb{E}\lambda_k-c_n(\sum_{k=1}^q \mathbb{E}(1-\lambda_k)^{-1}-q)+\frac{c_n^2}{n}\sum_{k=1}^q\mathbb{E}((1-\lambda_k)^{-2}-(1-\lambda_k)^{-1}). 
$$\end{lem}

\begin{proof} 
Lemma \ref{del} ensures the joint density function of entries of $\sqrt{n} {Z}_n$ is
$$
\aligned
    f_n(z) & =(\sqrt{2 \pi})^{-pq}n^{-\frac{pq}{2}} \frac{w\left(n-p, q\right)}{w(n,q)} \left\{\operatorname{det}\left(\mathbf I_q-\frac{z^{\prime}  z}{n}\right)^{c_n}\right\} I_0\left(\frac{z^{\prime}  z}{n}\right)\\ & =A_{n, p, q} \operatorname{det}\left(\mathbf I_q-\frac{z^{\prime} z}{n}\right)^{c_n} I_0\left(\frac{z^{\prime} z}{n}\right),
\endaligned$$
where $A_{n, p, q}:=(\sqrt{2 \pi})^{-pq}n^{-\frac{pq}{2}} \frac{w\left(n-p, q\right)}{w(n,q)}.$ 

Now the joint density function of entries of ${G}_n$ is 
$$
    g_n(z) = (\sqrt{2 \pi})^{-pq} e^{-\frac{{\rm tr}(z'z)}{2}},
$$
whence 
\begin{equation}\label{RNdensity}\frac{f_n}{g_n}(z)=B_{n, p, q} \operatorname{det}\left(\mathbf I_q-\frac{z^{\prime} z}{n}\right)^{c_n} e^{\frac{{\rm tr}(z'z)}{2}} I_0\left(\frac{z^{\prime} z}{n}\right).\end{equation}
It follows from definition that 
\begin{equation}\label{modiI} \aligned I(\mathcal{L}(\sqrt{n}{Z}_n)|\mathcal{L}({G}_n))
&= \frac14\int_{R^{pq}} |\nabla \frac{f_n}{g_n}|(z)^2 \frac{g_n^2}{f_n^2}(z) f_n(z) dz  \\
&= \frac14\mathbb{E}\big[\big|\nabla \frac{f_n}{g_n}\big|^2 \frac{g_n^2}{f_n^2}(\sqrt{n}{Z}_n)\big]\\
&=\frac14\mathbb{E}\big[\big|\nabla \log\frac{f_n}{g_n}\big|^2(\sqrt{n}{Z}_n)\big]. 
\endaligned \end{equation} 
The expression \eqref{RNdensity} leads 
$$\big|\nabla \log\frac{f_n}{g_n}\big|^2(z)=|\nabla (c_n\log {\rm det} (\mathbf I_q-\frac{z'z}{n})+\frac{1}{2}{\rm tr}(z'z))|^2,$$ whence 
\begin{equation}\label{form} \big|\nabla \log\frac{f_n}{g_n}\big|^2(z)=\sum_{i=1}^p\sum_{j=1}^q(\partial_{ij} (c_n\log {\rm det} (\mathbf I_q-\frac{z'z}{n})+\frac{1}{2}{\rm tr}(z'z))^2.\end{equation}
Here the notation $\partial_{ij}$ is the partial derivative with respect to the variable $z_{ij}$ for all $1\le i\le p$ and $1\le j\le q.$ 
Let $(x_1, \cdots, x_q)$ be the eigenvalues of $z'z.$ Now we are going to express the right hand side of \eqref{form} as a functional of $(x_i)_{1\le i\le q}.$
We always assume $I_0(\frac{z'z}{n})=1$, and set ${S}=z^{\prime} z,$ which means $s_{ij}=\sum_{\ell=1}^p z_{\ell i}z_{\ell j}$ for all $1\le i, j\le q.$
Thus, 
$$\aligned \label{simpli}
c_n\log {\rm det} (\mathbf I_q-\frac{z'z}{n})+\frac{1}{2}{\rm tr}(z'z)
&=\frac{1}{2}\displaystyle {{\rm tr} (S)}+
c_n\log {\rm det} (\mathbf I_{q}-n^{-1}  S)
\endaligned 
$$ 
and for the target \eqref{form}, the partial derivatives $\partial_{ij} {{\rm tr} (S)}$ and $\partial_{ij} {\rm det} (\mathbf I_{q}-n^{-1} S)$ are in need. 

First, we get
\begin{equation*}
\partial_{ij}s_{kk}=
\begin{cases} 
0,&k\neq j;\\
2z_{ij},&k=j,
\end{cases}
\end{equation*}
and then
\begin{equation}\label{partials}
\partial_{ij} {\rm tr}(S)=2 z_{ij}. 
\end{equation}

For simplicity, for any $m\times q$ matrix $A,$ we write $A=(A_1, \cdots, A_q)$ in this part. Set $D=\mathbf I_q-\frac{1}{n}S$ and clearly, $D_{\ell}=e_{\ell}-n^{-1}S_{\ell}.$ It is well known that the determinant of $ D$ satisfies  
\begin{align*}
&\partial_{ij}{\rm det} (D)=\sum_{\ell=1}^q {\rm det}(\left(D_1,  \cdots, D_{\ell-1}, \partial_{ij} D_{\ell}, D_{\ell+1}, \cdots, D_{q}\right))
\end{align*}
with $\partial_{ij}D_{\ell}:=(\partial_{ij} D_{1, \ell}, \cdots, \partial_{ij}D_{q, \ell})'.$
Now \begin{equation*}
\partial_{ij}D_{\ell} 
= \partial_{ij} \left( e_{l} - n^{-1} S_{l} \right)
= -\frac{1}{n} \partial_{ij} S_{l}=
\begin{cases}
-\frac{1}{n} z_{ij} e_j- \frac{1}{n} (z')_{i}, & l = j; \\
-\frac{1}{n} z_{il} e_{j}, & l \neq j,
\end{cases}
\end{equation*}
which is equivalent to say 
$$\partial_{ij}D_{\ell}=-\frac{1}{n}z_{il}e_j-\frac{1}{n}(z')_i\delta_{\ell=j}$$ with  
$$(z')_{i}=(z_{i1}, \cdots, z_{iq})'.$$ 
Thus, 
\begin{align*}
\partial_{i j} {\rm det}(D)
=&{\rm det}(\left(\partial_{ij} D_1, D_2, \cdots, D_q\right))+\cdots+{\rm det}(\left(D_1, D_2, \cdots, \partial_{i j}D_q\right))\\
=&\sum_{\ell=1}^q {\rm det}\left((D_1, \cdots, D_{\ell-1}, -\frac1n z_{il}e_j, D_{l+1}, \cdots, D_q)\right)\\
&\quad+{\rm det}\left((D_1, \cdots, D_{j-1}, -\frac1n z_{i}', D_{j+1}, \cdots, D_q)\right).
\end{align*}
Set $$L_{i j}:={\rm det}\left((D_1, \cdots, D_{j-1}, -\frac1n z_{i}', D_{j+1}, \cdots, D_q)\right)$$ and we claim that 
\begin{equation}\label{lij}\sum_{\ell=1}^q {\rm det}\left((D_1, \cdots, D_{\ell-1}, -\frac1n z_{il}e_j, D_{l+1}, \cdots, D_q)\right)=L_{ij}.\end{equation}
Once the expression \eqref{lij} is true, we see  
\begin{equation}\label{partialD} 
\partial_{i j} {\rm det}(D)=2 L_{i, j}.
\end{equation}	
Indeed, expanding the determinant $${\rm det}\left((D_1, \cdots, D_{\ell-1}, -\frac1n z_{il}e_j, D_{l+1}, \cdots, D_q)\right)$$  along the $\ell$-th column, one obtains 
$$\aligned 
&\quad \sum_{\ell=1}^q {\rm det}\left((D_1, \cdots, D_{\ell-1}, -\frac1n z_{il}e_j, D_{l+1}, \cdots, D_q)\right)\\
&=\sum_{\ell=1}^q(-\frac1n z_{i \ell}) (-1)^{\ell+j}(D)_{j l}\\
&={\rm det}\left((D_1, \cdots, D_{\ell-1}, -\frac1n z_{i}', D_{l+1}, \cdots, D_q)\right),
\endaligned $$
which is exactly the claiming equation \eqref{lij} and where $(D)_{j l}$ is the $(j, l)$-th
 minor of $D.$

Putting \eqref{partials} and \eqref{partialD} back into \eqref{form}, one gets 
\begin{equation}\label{formsum}\aligned 
\big|\nabla \log\frac{f_n}{g_n}\big|^2(z)&=\sum_{i=1}^p\sum_{j=1}^q\left(c_n({\rm det}(D))^{-1}\partial_{i j} {\rm det}(D)+ \frac12\partial_{ij} {\rm tr}(S)\right)\\
&=\sum_{i=1}^p\sum_{j=1}^q(2c_n({\rm det}(D))^{-1}L_{ij}+z_{ij})^2\\
&=\sum_{i=1}^p\sum_{j=1}^q(z_{ij}^2+4c_n^2({\rm det}(D))^{-2}L_{ij}^2+4c_n({\rm det}(D))^{-1}L_{ij}z_{ij})\\
&=\sum_{k=1}^q x_k+4c_n^2({\rm det}(D))^{-2} M_1+4c_n({\rm det}(D))^{-1}M_2,
\endaligned
\end{equation}
where the last equality holds because 
$$\sum_{i=1}^p\sum_{j=1}^qz_{ij}^2={\rm  tr}(S)=\sum_{i=1}^q x_i$$ and  
$$M_1:=\sum_{i=1}^p\sum_{j=1}^qL_{ij}^2\quad \text{and} \quad \displaystyle M_2:=\sum_{i=1}^p\sum_{j=1}^q L_{ij}z_{ij}.$$ We continue to find the expression of $M_1$ and $M_2$ via $(x_i)_{1\leq i \leq q}.$

Let $\xi_i=\left( \xi_{i1}, \xi_{i2}, \dots, \xi_{iq} \right)'$ be the solution to the linear equations $D \xi=-\frac{1}{n}(z')_i.$ The condition $I_0(z'z/n)$ guarantees that $D$ is invertible and  then the Cramer's formula for linear equations tells 
\begin{equation*}
    \xi_{ij}=\frac{L_{ij}}{{\rm det}(D)}. 
\end{equation*}

Hence,
\begin{align*}
\displaystyle \sum_{j=1}^q L_{ij}^2
&={\rm det}^2(D)\sum_{j=1}^q \xi_{ij}^2\\
&={\rm det}^2(D)
[-\frac{1}{n}D^{-1}(z')_i]'(-\frac{1}{n}{D}^{-1}(z')_i)\\
&=\frac{1}{n^2}{\rm det}^2(D) ((z')_i)' 
D^{-2}(z')_i,
\end{align*}
and consequently 
\begin{align*}
\displaystyle M_1
&=\frac{1}{n^2}{\rm det}^2(D) \sum_{i=1}^q ((z')_i)' 
D^{-2} (z')_i\\
&=\frac{1}{n^2}{\rm det}^2(D)
{\rm tr}(z D^{-2}z')\\
&=\frac{1}{n^2}{\rm det}^2(D)
{\rm tr}( D^{-2} S ).
\end{align*}

Since $S$ is real symmetric, there exists an orthogonal matrix $P$ such that $S=P'{\Lambda} P$, where ${\Lambda}={\rm diag}(x_1, \cdots, x_n).$
It follows again from the cyclic property of trace that  
$$\aligned {\rm tr}( D^{-2} S)&={\rm tr}( D^{-1} S D^{-1})\\
&={\rm tr}\left(P'( {\mathbf I}_q-n^{-1} {\Lambda} )^{-1} P( P'{\Lambda}  P) P'( {\mathbf I}_q-n^{-1} {\Lambda} )^{-1}P\right )\\
&={\rm tr}\left(( {\mathbf I}_q-n^{-1} {\Lambda} )^{-1}{\Lambda} ( {\mathbf I}_q-n^{-1} {\Lambda} )^{-1}\right )\\
&=\sum_{k=1}^q \frac{x_k}{(1-\frac{1}{n}x_k)^2},
\endaligned $$
whence 
\begin{equation}\label{M1} 
	M_1=\frac{1}{n^2}{\rm det}^2(D)\sum_{k=1}^q \frac{x_k}{(1-\frac{1}{n}x_k)^2}.
\end{equation}
Note that
\begin{align*}
z_{ij}L_{ij}
&= {\rm det}\left((D_1, \cdots, D_{j-1}, -\frac1n z_{ij} (z')_{i}, D_{j+1}, \cdots, D_q)\right)\\
&=\sum_{k=1}^q(-1)^{j+k}(-\frac1n z_{i j}z_{ik})(D)_{k j}\\
\end{align*}
and then
\begin{align*}
\displaystyle \sum_{i=1}^p z_{ij}L_{ij}
&=\sum_{k=1}^q (-1)^{j+k}(-\frac1n s_{kj})(D)_{k j}\\
&={\rm det}\big((D_1, \cdots, D_{j-1}, -\frac1n s_{j}, D_{j+1}, \cdots, D_q)\big)\\
&={\rm det}(D)-{\rm det}\big((D_1, \cdots, D_{j-1}, e_j,  D_{j+1}, \cdots, D_q)\big)\\
&={\rm det}(D)-(D)_{jj}.
\end{align*}
It follows  
\begin{align*}
\displaystyle \sum_{j=1}^q \sum_{i=1}^p z_{ij}L_{ij}
&=q\; {\rm det}(D)
-
\sum_{j=1}^q (D)_{jj}=(q-\sum_{k=1}^q \frac{1}{1-\frac{1}{n}x_k})\prod_{k=1}^q (1-\frac{1}{n}x_k),
\end{align*}
where we use two facts 
$${\rm det}(D)=\prod_{k=1}^q (1-n^{-1} x_k) \quad \text {and}\quad \sum_{j=1}^q (D)_{jj}={\rm det}(D){\rm tr}(D^{-1}).$$
Eventually, one gets 
\begin{equation} \label{M2}
M_2=(q-\sum_{k=1}^q \frac{1}{1-\frac{1}{n}x_k})\prod_{k=1}^q (1-\frac{1}{n}x_k).
\end{equation}
Plugging \eqref{M1} and \eqref{M2} into \eqref{formsum}, we get 
\begin{align*}
\big|\nabla \log\frac{f_n}{g_n}\big|^2(z)
&=\sum_{k=1}^q x_k
+4c_n^2\sum_{k=1}^q \frac{x_k}{(n-x_k)^2}
-4c_n \sum_{k=1}^q \frac{x_k}{n-x_k}\\
%&=\sum_{k=1}^q x_k \left( \frac{x_k - p - q - 1}{n - x_k} \right)^2.
&=4c_n q+\sum_{k=1}^qx_k-4c_n(c_n+n)\sum_{k=1}^q (n-x_k)^{-1}+4n c_n^2\sum_{k=1}^q(n-x_k)^{-2}.
\end{align*}
Reviewing \eqref{modiI}, we have 
\begin{align*}
I(\mathcal{L}(\sqrt{n}{Z}_n)|\mathcal{L}(G_n))
=\frac{n}4\sum_{k=1}^q\mathbb{E}\lambda_k-c_n(\sum_{k=1}^q \mathbb{E}(1-\lambda_k)^{-1}-q)+\frac{c_n^2}{n}\sum_{k=1}^q\mathbb{E}((1-\lambda_k)^{-2}-(1-\lambda_k)^{-1}).
\end{align*}
The proof  is finally completed. 
\end{proof}

\subsection{Joint density function of  $(\lambda_k)_{1\le k\le q}$.}

\quad\quad\quad\quad\quad\quad\quad\quad\quad\quad\quad\quad
\\

Lemma \ref{expressofI} offers an explicit expression of the Fisher information via $(\lambda_k)_{1\le k\le q}.$ We next give their joint density function. The parallel result for the unitary case is given in \cite{Forrester06}. 
\begin{lem}\label{jointdf} 
Let $(\lambda_k)_{1\le k\le q}$ be eigenvalues of $Z'_n Z_n.$ The joint density function of $(\lambda_1, \cdots, \lambda_q)$ is proportional to 
\begin{equation}\label{denlambda}
	\prod_{i=1}^q (1-x_i)^{\frac{n-p-q-1}{2}}x_{i}^{\frac{p-q-1}{2}} \prod_{1\le i\neq j\le q}|x_i-x_j|.
\end{equation}
\end{lem}

\begin{proof} 
The first step is to obtain the joint density of $Z_n'Z_n.$ 
Recall the joint density function of entries of $Z_n$ is
\begin{equation*}
f(z)=(\sqrt{2\pi})^{-pq}\frac{\omega(n-p, q)}{\omega(n, q)}\left\{{\rm det}(\mathbf I_{q}-z'z)^{(n-p-q-1)/2}\right\}I_0(z'z).
\end{equation*}

The matrix $Z_n'Z_n $ is symmetric, there exists a  lower triangular matrix $T$ with $t_{ii}>0$ and  a semi-orthogonal matrix with $H_1 H_1' = \mathbf{I}_q$ such that $Z_n' =T H_1. $  Let 
$H
=\begin{pmatrix}
 H_1 \\ H_2
\end{pmatrix}$
 be an orthogonal matrix. The formula $(1.3.25)$ in \cite{matrix} tells us that the Jacobian of this transformation is  
\begin{equation*}
J({Z}_n' \rightarrow ({T}, {H_1})) =  \prod_{i=1}^{q} t_{ii}^{p-i}g_{p,q}(H_1),
\end{equation*}
where $g_{p,q}(H_1)=J((d H_1) H'\rightarrow(d H_1)) $.

Hence, the joint density of $( T,  H_1)$ is:
$$
 (\sqrt{2 \pi})^{-pq} \frac{\omega\left(n-p, q\right)}{\omega(n,q)} {\rm{det}}\left(\mathbf{I}_q-{T} {T}^{\prime}\right)^{\frac{n-p-q-1}{2}}  \prod_{i=1}^{q} t_{ii}^{p-i} ({H}_1)g_{p,q}( H_1)I_0\left({T} {T}^{\prime}\right) .
$$
Now, Theorem 1.4.9 of \cite{matrix} ensures the marginal density of $T$ is
$$
  \frac{2^q \omega\left(n-p, q\right)\omega(p, q)}{\omega(n, q)}
 \left({\rm det}(\mathbf{I}_q-{T} {T}^{\prime})\right)^{\frac{n-p-q-1}{2}}
 I_0\left({T} {T}^{\prime}\right) 
 \prod_{i=1}^{q} t_{ii}^{p-i}.
$$
Since $S=Z_n'Z_n=T T'$ and the Jacobian $J(T\rightarrow T T')=(2^q\prod_{i=1}^q t_{ii}^{q-i+1})^{-1}$, the joint density function of $Z_n' Z_n$ turns out to be  
\begin{equation}\label{matrixs}
\frac{\omega(n-p, q) \omega(p, q)}{\omega(n,q)} 
{\rm det}\left(\mathbf{I}_q-S\right)^{\frac{n-p-q-1}{2}}
   \operatorname{det}(S)^{\frac12(p-q-1)}I_0\left(S\right).
\end{equation}

Next, we derive the joint density function of $(\lambda_1,\dots,\lambda_q)$ from that of $Z_n'Z_n$. Let $h$ be the joint density function of $S.$ Since both ${\rm det}({\mathbf I} -S)$ and $\rm{det}(S)$ are symmetric polynomials of $(\lambda_k)_{1\leq k \leq q},$ there exists a function $\phi$ such that $h=\phi(\sigma_1,\dots,\sigma_q)$, where $\sigma_k(\lambda_1,\dots,\lambda_q)=\displaystyle\Sigma_{1\leq i_1 < \dots < i_k \leq q}\lambda_{i_1}\dots \lambda_{i_k}$ (see the Theorem 3.1.1 in \cite{polynomials}). According to the Newton formula (see 3.3 of \cite{polynomials}), there exists a function $\Phi$ such that 
$$
h(S)=g(\sigma_1,\dots,\sigma_q)=\Phi(s_1,\dots,s_q),
$$
where $s_l=\displaystyle \sum_{k=1}^q \lambda_k^l$. Note that $s_l={\rm tr} (S^l)$ and then we have
$$
h(S)=\Phi({\rm tr}( S^1), \dots, {\rm tr}( S^q)).
$$
It follows from $(6.16)$ in \cite{intr to mat} and \eqref{matrixs} that the joint density function of $(\lambda_1,\dots,\lambda_q)$ is proportional to
$$
  \prod_{i=1}^q (1-x_i)^{\frac{n-p-q-1}{2}}x_{i}^{\frac{p-q-1}{2}} \prod_{1\le i\neq j\le q}|x_i-x_j|.
$$
The proof is then completed. 
\end{proof}

\subsection{The expectations $\mathbb{E}(1-\lambda_i)^{-k}$ for $k=1, 2$}

\quad \quad\quad \quad\quad \quad\quad \quad\quad \quad\quad \quad \quad \quad\quad \quad\quad \quad\quad \\
\quad 
 Observing the joint density function of $(\lambda_1, \cdots, \lambda_q)$ obtained in Lemma \ref{denlambda},  $(\lambda_1, \cdots, \lambda_q)$ is recognized as a Jacobi ensemble.  
 Lemma 3.1 in \cite{Ma25} offers exactly the expectations $\mathbb{E}(1-\lambda_i)^{-k}$ for $k=1, 2,$  which was established based on the tridiagonal random matrix characterization of $\beta$-Jacobi ensembles (see \cite{ES08}). But the results with tails $o(\frac{pq^2}{n^2})$ are not enough for our use because we need more details on the expectations. Thus, we follow the idea of the proof of Lemma 3.1 in \cite{Ma25} with more patience to get the following expectations.    
\begin{lem}\label{inversemj}  Let $\boldsymbol{\lambda}=(\lambda_1, \cdots, \lambda_n)$ be the eigenvalues of $Z_n'Z_n$. Suppose that $1\le q\le p$ and
$p=o(n).$ We have that
$$
\aligned 
\mathbb{E}\sum_{i=1}^n \frac{1}{1-\lambda_i}&=q+\frac{pq}{n-p}+\frac{pq(q+1)}{n^2}+\frac{2p^2q(q+1)+pq^3}{n^3}+o\left(\frac{p^2q^2}{n^3}\right); \\
\mathbb{E}\sum_{i=1}^n \frac{1}{(1-\lambda_i)^2}
	&=q+\frac{2pq}{n-p}+\frac{p^2q}{(n-p)^2}+\frac{3pq(q+1)}{n^2}+\frac{9p^2q(q+1)+4 pq^3}{n^3}+o(\frac{p^2q^2}{n^3})\endaligned 
$$
as $n$ sufficiently large. 
\end{lem} 

\begin{proof}
The joint density function of $(1-\lambda_i)_{1\le i\le q}$ is still a Jacobi ensemble.     
Edelman-Sutton's characterization implies that   $(1-\lambda_i)_{1\le i\le q}$ are the eigenvalues of the random matrix ${D' D},$
where the $q\times q$ random matrix ${D}$ has the following form \begin{align}\label{matrixb}
\bf D=
    \begin{pmatrix}
      x_1&  &  &   &  \\
      y_2 & x_2 &   &   &\\
      &    \ddots &   \ddots & & \\
       &  & y_{q-1} & x_{q-1} &   \\
      &   &   &   y_q & x_q
    \end{pmatrix}
\end{align}
with $x_i=\sqrt{c_i c_{i+1}'}$ for $1\le i\le q$ and 
$y_i=-\sqrt{s_i s_i'}$ for $2\le i\le q.$ 
Here the non-negative random variables $c_i, i=1, 2, \cdots, q$ and $s_i, c'_i,  s'_i, i=2, 3, \cdots, q$ obeying the distribution and relationships as 
\begin{itemize}
\item[1).] $\{c_1, c_2, \cdots, c_q, c'_2, c'_3, \cdots, c'_{q}\}$ independent;
\item[2).] $c_i\sim {\rm Beta}(\frac12(n-p-i+1), \frac12(p-i+1)), \; 1\le i\le q;$
\item[3).] $c'_i\sim{\rm Beta}(\frac12(n-q-i+2), \frac12 (q-i+1)), \; 2\le i\le q;$
\item[4).] $s_i+c_i=1, \quad s'_i+c'_i=1, \; 2\le i\le q;$
\item[5).] $c'_{n+1}=1.$ 
\end{itemize}
By Lemma 5.1 in \cite{Ma25}, $((1-\lambda_i)^{-1})_{1\le i\le q}$ can be regarded as the eigenvalues of the matrix ${V' V},$ with entires satisfying
 $$
v_{i, j}=\begin{cases} \frac{1}{x_i}, &  1\le i=j\le n; \\
(-1)^{i+j} \frac{\prod_{k=j+1}^i y_k}{\prod_{k=j}^i x_k}, & 1\le j\le i-1\le n-1. \\ 
0, \quad & \text{otherwise}.
\end{cases}
$$
Hence,  
\begin{equation}\label{square}
	v_{i, j}^2=\frac{1}{c_{i+1}'c_{j}}\prod_{l=j+1}^i \frac{s_ls_l'}{c_l \, c_l'}{\bf 1}_{1\le j\le i-1}, \quad \text{and}\quad v_{i, i}^2=\frac{1}{c_i c_{i+1}'}\end{equation} for 
$1\le i\le q.$ 
First, for a random variable $\xi\sim {\rm Beta}(\alpha, \beta),$ one knows 
\begin{equation}\label{betad} 
	\mathbb{E} \frac{1}{\xi}=1+\frac{\beta}{\alpha-1}; \quad \mathbb{E} \frac{1}{\xi^2}=1+\beta(\frac{1}{\alpha-1}+\frac{1}{\alpha-2})+\frac{\beta^2}{(\alpha-1)(\alpha-2)};
\end{equation}
Before we go further, we make a little explanation on what terms we can put away. In fact, since our result stop at $o(\frac{p^2q^2}{n^3})$ and only $q\le p=o(n)$ are in need, we can only ignore terms like $$\frac{pq^2+pq+q^2+q^3}{n^3}$$ but not $\frac{p^2+p^2q}{n^3}$ because $q$ could be bounded and we try  to avoid the appearance of the term $\frac{p^3+p^3q}{n^3}$. Based on this understanding, we do the sequent calculus.  

Utilize repeatedly the decomposition 
$$\frac1{n-p-i}=\frac{1}{n}(1+\frac{p+i}{n-p-i})
$$ 
to get 
\begin{equation}\label{decom} 
	\frac1{n-p-i}=\frac{1}{n-p}+\frac{i}{(n-p)^2}+\frac{i^2}{(n-p)^3}+O(\frac{q^3}{n^4})
\end{equation}
and similarly 
\begin{equation}\label{decom1} 
	\frac1{n-q-i}=\frac{1}{n}+\frac{q+i}{n^2}+\frac{(q+i)^2}{n^3}+O(\frac{q^3}{n^4})
\end{equation}
for $1\le i\le q.$ We will also use frequently the following asymptotic to simplify the expressions 
\begin{equation}\label{cruciala}\frac{1}{(n-p)n}=\frac{1}{n^2}+\frac{p}{n^3}+O(\frac{p^2}{n^4}), \quad  \frac{1}{(n-p)^2}=\frac{1}{n^2}+\frac{2p}{n^3}+O(\frac{p^2}{n^4}) \quad \text{and}\quad \frac{1}{(n-p)^3}=\frac{1}{n^3}+O(\frac{p}{n^4}).\end{equation}

By \eqref{betad} and the independence of $c_i$ and $s_i$, together with  \eqref{decom}, \eqref{decom1} and \eqref{cruciala}, it follows 
 that 
\begin{equation}\label{keyc} 
	\aligned \mathbb{E} \frac{s_i}{c_i}&=\frac{p-i+1}{n-p-i-1}\\
	&=\frac{p-i+1}{n-p}+\frac{(p-i+1) (i+1)}{ (n-p)^2}+\frac{(p-i+1) (i+1)^2}{ (n-p)^3}+O\left(\frac{pq^3}{n^4}\right);\\
	\mathbb{E} \frac{s_i'}{c_i'}&=\frac{q-i+1}{
	n-q-i}\\
	&=\frac{q-i+1}{n}+\frac{(q-i+1)(q+i)}{n^2}+\frac{(q-i)(q+i)^2}{n^3}+O(\frac{q2}{n^3}). \endaligned  \end{equation} 
 
Thus, for 
$1\le j\le i-2\le q-2,$ 
one gets from \eqref{keyc}, independence and the relationships $$\frac 1{c_i}=1+\frac{s_i}{c_i}=O(1), \quad \frac 1{c_i'}=1+\frac{s_i'}{c_i'}=O(1)$$ that  
$$\mathbb{E} v_{i, j}^2=\mathbb{E}\frac{1}{c_{i+1}'c_{j}}\prod_{l=j+1}^i \frac{s_ls_l'}{c_l \, c_l'}=O\left(\frac{pq}{n^2}\right)^{i-j}.$$
Consequently, 
$$\sum_{1\le j\le i-2\le q}\mathbb{E} v_{i, j}^2=O\left(\frac{p^2q^3}{n^4}\right)=o(\frac{p^2q^2}{n^3}).$$ 
Therefore 
$$\sum_{i=1}^q(1-\lambda_i)^{-1}=\mathbb{E}\sum_{1\le i, j\le n} v_{i, j}^2=\mathbb{E}\sum_{i=1}^q v_{i, i}^2+\mathbb{E}\sum_{i=1}^{q-1} v_{i+1, i}^2+o\left(\frac{p^2q^2}{n^3}\right).$$
We first check the sum $\mathbb{E}\sum_{i=1}^{q-1} v_{i+1, i}^2.$ Indeed, by the definition of $v_{i+1, i},$ we see 
$$\mathbb{E} v_{i+1, i}^2=\mathbb{E}\frac{1}{c_{i+2}'}\;\mathbb{E}\frac{1}{c_i}\;\mathbb{E}\frac{s_{i+1}}{c_{i+1}}\;\mathbb{E}\frac{s_{i+1}'}{c_{i+1}'}.$$ Keep in mind that we only pay close attention to the terms with order greater than $O(\frac{p^2q}{n^3})$ and observe that  
$$\mathbb{E}\frac{s_{i+1}'}{c_{i+1}'}\frac{s_{i+1}}{c_{i+1}}=O(\frac{pq}{n^2}) \quad \text{and}\quad \mathbb{E}\frac{1}{c_{i+2}'}\;\mathbb{E}\frac{1}{c_i}=O(1).$$ Applying \eqref{cruciala} and \eqref{keyc} to proper terms while ignoring the terms $o(\frac{p^2 q^2}{n^3})$, we write 
$$\aligned & \quad \mathbb{E}\frac{1}{c_{i+2}'}\;\mathbb{E}\frac{1}{c_i}\;\mathbb{E}\frac{s_{i+1}'}{c_{i+1}'}\mathbb{E}\frac{s_{i+1}}{c_{i+1}}\\
&=(1+\frac{q-i-1}{n}+O(\frac{q^2}{n^2}))\times(1+\frac{p-i+1}{n-p}+O(\frac{p^2}{n^2}))\\
&\quad\times \left(\frac{p-i}{n-p}+\frac{(p-i) i}{ (n-p)^2}+O(\frac{p}{n^2})\right)\times\left(\frac{q-i}{n}+\frac{(q-i)(q+i)}{n^2}+O(\frac{q}{n^2})\right)\\
&=\frac{(q-i)(p-i)}{n(n-p)}(1+\frac{p+q-2i}{n})(1+\frac{q+2i}{n})+o(\frac{p^2q}{n^3})\\
&=\frac{(q-i)(p-i)}{n^2}(1+\frac{2(p+q)}{n})+o(\frac{p^2q}{n^3}).
\endaligned $$ Now 
$$\sum_{i=1}^{q-1}(q-i)(p-i)=\frac16q(q-1)(3p-q-1),$$ one gets 
\begin{equation}\label{vi1}
\aligned
 \sum_{i=1}^{q-1}\mathbb{E} v_{i+1, i}^2&= \frac{pq(q+1)}{2n^2}-\frac{q(q-1)(q+1)}{6n^2}+\frac{(p+q)pq(q-1)}{3n^3}-\frac{(p+q)q^3}{3n^3}+o(\frac{p^2q^2}{n^3}).
	\endaligned
\end{equation}
For $\mathbb{E} v_{i, i}^2=\mathbb{E}\frac{1}{c_{i+1}'}\;\mathbb{E}\frac{1}{c_i},$ we have to write these two expectations until $O(n^{-4})$ to meet the final order $o(\frac{p^2q^2}{n^3})$ for the sum because both $\mathbb{E}\frac{1}{c_{i+1}'}$ and $\mathbb{E}\frac{1}{c_i}$ are of order $O(1).$  Thereby, with the convention $c_{q+1}'=1$, combining alike terms according to the coefficient of $n^{-2}$ and $n^{-3}$ and continuing to ignore $o(\frac{p^2q^2}{n^3}),$ we are able to get  $$
	\aligned \mathbb{E} \sum_{i=1}^{q} v_{i, i}^2&=\sum_{i=1}^{q}\mathbb{E}\frac{1}{c_{i+1}'}\;\mathbb{E}\frac{1}{c_i}\\
	&=\sum_{i=1}^{q}\left(1+\frac{q-i}{n}+\frac{(q-i)(q+i+1)}{n^2}+\frac{(q-i)(q+i)^2}{n^3}\right)\\
	&\quad \times \left(1+\frac{p-i+1}{n-p}+\frac{(p-i+1) (i+1)}{ n^2}+\frac{(p-i) (i+1)(2p+i)}{n^3}\right)+O(\frac{pq}{n^3})
	\\
	&=p+\frac{pq}{n-p}+\frac{q(q-1)(q+1)}{6n^2}+\frac{pq(q+3)}{2n^2}+\frac{1}{n^3}(\frac{1}3(p+q)q^3+p^2q(q+3))+o(\frac{p^2q^2}{n^3}).\endaligned 
$$
Summing these asymptotic expressions together, we are ready to claim  
$$\aligned \mathbb{E}\sum_{i=1}^q \frac1{1-\nu_i}&=\sum_{i=1}^q\mathbb{E} v_{i, i}^2+\sum_{i=1}^{q-1}\mathbb{E} v_{i+1, i}^2+o\left(\frac{p^2q^2}{n^3}\right)\\
&=p+\frac{pq}{n-p}+\frac{pq(q+1)}{n^2}+\frac{1}{n^3}(2p^2q(q+1)+p q^3)+o\left(\frac{p^2q^2}{n^3}\right). \endaligned    
$$ 
As we can see from the proof of Lemma 3.1 in \cite{Ma25}, 
 \begin{equation}\label{sum00}
	 \mathbb{E} \sum_{i=1}^n \frac{1}{(1-\nu_i)^2}=\sum_{i=1}^q \mathbb{E} v_{i, i}^4+2\sum_{i=1}^{q-1}\mathbb{E} (v_{i, i}^2 v_{i+1, i}^2+v_{i+1, i}^2 v_{i+1, i+1}^2)+o(\frac{p^2q^2}{n^3}). \end{equation} 
	 It follows from the definition of $v_{i, j}$ that
	$$
	 \aligned  
	   v_{i, i}^4&=\frac{1}{c_{i}^2}\frac{1}{c_{i+1}'^2}, \quad   v_{i, i}^2 v_{i+1, i}^2=\frac{1}{c_i^2c_{i+2}'}\frac{s_{i+1}}{c_{i+1}}\frac{s_{i+1}'}{c_{i+1}'^2}\quad\text{and} \quad v_{i+1, i+1}^2 v_{i+1, i}^2=\frac{1}{c_ic_{i+2}'^2}\frac{s_{i+1}}{c_{i+1}^2}\frac{s_{i+1}'}{c_{i+1}'} 
	 \endaligned 	
	$$
We prepare two similar expressions via  \eqref{betad} as \eqref{keyc}, which say
\begin{equation}\label{keytm0}
	\aligned  \mathbb{E} \frac{1}{c_i^2}&=1+\frac{(p-i+1)}{n-p}\left(2+\frac{p+i+5}{n-p}\right)+\frac{2p(i+2)(p-i)}{n^3}+O(\frac{pq}{n^3});\\
	 \mathbb{E} \frac{1}{c_i'^2}&=1+\frac{q-i+1}{n}\left(2+\frac{3q+i+3}{n}+\frac{4q^2+4qi}{n^2}\right)+O\left(\frac{q^2}{n^3}\right)\endaligned  \end{equation}	
	for $1\le i\le q.$ 
Thus, it follows from \eqref{keytm0} that 
$$\aligned  \mathbb{E} v_{i, i}^4&=\mathbb{E}\frac{1}{c_{i}^2}\mathbb{E} \frac{1}{c_{i+1}'^2}\\
&=1+\frac{2(p-i+1)}{n-p}\left(2+\frac{p+i+5}{n-p}\right)+\frac{2p(i+2)(p-i)}{n^3}+O(\frac{pq}{n^3})\\
&\quad +\frac{q-i}{n}\left(2+\frac{3q+i+4}{n}+\frac{4q^2+4qi}{n^2}\right)\\
&\quad +\frac{(q-i)(p-i+1)}{n(n-p)}\left(2+\frac{p+i+5}{n-p}\right)\times\left(2+\frac{3q+i+4}{n}\right)\\
&=1+\frac{(p-i+1)}{n-p}\left(2+\frac{p+i+5}{n-p}+\frac{2p(i+2)}{(n-p)^2}\right)+\frac{2(q-i)}{n}+\frac{(q-i)(4p+3q+8-3i)}{n^2}\\
&\quad+\frac{2(q-i)(2q^2+3pq+3p^2-(p+q)i-2i^2)}{n^3}+O(\frac{pq}{n^3}).\endaligned $$ 
Observing $$\aligned\sum_{i=1}^q\frac{2p(p-i+1)(i+2)}{(n-p)^3}&=\frac{2p}{n^3}\sum_{i=1}^q(p-i+1)(i+2)+O(\frac{p^3q^2}{n^4})\\
&=\frac{p^2q(q+5)}{n^3}-\frac{2pq^3}{3n^3}+o(\frac{p^2q^2}{n^3})\endaligned $$ 
and $$\sum_{i=1}^q(\frac{2(p-i+1)}{n-p}+\frac{2(q-i)}{n})=\frac{2pq}{n-p}-\frac{pq(q-1)}{n^2}-\frac{p^2q(q-1)}{n^3}+o(\frac{p^2q^2}{n^3}),$$ then with some basic calculus, one gets 
\begin{equation}\label{4di}\aligned  &\quad\sum_{i=1}^q\mathbb{E} v_{i, i}^4\\
&=q+\frac{2pq}{n-p}+\frac{q(q-1)(p+q+\frac72)}{n^2}+\frac{q}{(n-p)^2}(p^2+6p-\frac52q-\frac13 q^2+\frac{17}6)\\
&\quad+\frac{1}{n^3}(3p^2q(q+1)+2pq^3+\frac{4}{3}q^4)+o(\frac{p^2q^2}{n^3}).\endaligned \end{equation}
Meanwhile, $$\sum_{i=1}^{q-1}\mathbb{E} v_{i, i}^2 v_{i+1, i}^2=\sum_{i=1}^{q-1}\mathbb{E}\frac{1}{c_i^2c_{i+2}'}\frac{s_{i+1}}{c_{i+1}}\frac{s_{i+1}'}{c_{i+1}'^2}.$$ 
It follows from \eqref{keyc} and \eqref{keytm0} that 
$$\aligned \mathbb{E}\frac{s_{i+1}'}{c_{i+1}'^2}&=\mathbb{E}(\frac{1}{c_{i+1}'^2}-\frac1{c_{i+1}'})=\frac{q-i}{n}+\frac{(q-i)(2q+3)}{n^2}
\endaligned $$
and similarly, as for $\sum_{i=1}^{q-1}\mathbb{E} v_{i+1, i}^2,$ since $\mathbb{E}\frac{s_{i+1}}{c_{i+1}}\frac{s_{i+1}'}{c_{i+1}'^2}=O(\frac{pq}{n^2}),$ to keep the terms with order not less than $\frac{p^2q^2}{n^3},$ we can simplify the expression of $\mathbb{E} v_{i, i}^2 v_{i+1, i}^2$ as 
$$\aligned &\quad \mathbb{E} v_{i, i}^2 v_{i+1, i}^2\\
&=\mathbb{E}\frac{1}{c_i^2c_{i+2}'}\frac{s_{i+1}}{c_{i+1}}\frac{s_{i+1}'}{c_{i+1}'^2}\\
&=(1+\frac{2(p-i+1)}{n-p})\times(1+\frac{q-i-1}{n})\times \left(\frac{p-i}{n-p}+\frac{(p-i) (i+2)}{ (n-p)^2}\right)\\
&\quad\times\left(\frac{q-i}{n}+\frac{(q-i)(2q+3)}{n^2}\right)+O(\frac{pq}{n^3})\\
&=O(\frac{pq}{n^3})+\frac{(q-i)(p-i)}{n(n-p)}(1+\frac{2p+q-3i}{n})\times\bigg(1+
\frac{2q+i}{n}\bigg)\\
&=O(\frac{pq}{n^3})+\frac{(q-i)(p-i)}{n^2}+\frac{(q-i)(p-i)(3p+3q-2i)}{n^3},
\endaligned $$ 
where for the last equality we use the fact 
$$\frac{(q-i)(p-i)}{n(n-p)}=\frac{(q-i)(p-i)}{n^2}+\frac{p(q-i)(p-i)}{n^3}+O(\frac{p^3q}{n^4}).$$
Thus, 
\begin{equation}\label{2cross2} \aligned \sum_{i=1}^{q-1}\mathbb{E} v_{i, i}^2 v_{i+1, i}^2&=o(\frac{p^2q^2}{n^3})+\sum_{i=1}^{q-1}\frac{(q-i)(p-i)}{n^2}(1+\frac{3p+3q-2i}{n})\\
&=\frac{q(q-1)(3p-q-1)}{6n^2}+\frac{1}{n^3}(\frac{3}{2}p^2q(q-1)+\frac{2}{3}pq^3-\frac13q^4)+o(\frac{p^2q^2}{n^3}).
\endaligned 
\end{equation}
Since $\mathbb{E} v_{i+1, i+1}^2 v_{i+1, i}^2=\mathbb{E}\frac{1}{c_ic_{i+2}'^2}\frac{s_{i+1}}{c_{i+1}^2}\frac{s_{i+1}'}{c_{i+1}'},$ the same analysis  brings us 
$$\aligned &\quad \mathbb{E} v_{i+1, i+1}^2 v_{i+1, i}^2
%&=\mathbb{E}\frac{1}{c_ic_{i+2}'^2}\frac{s_{i+1}}{c_{i+1}}\frac{s_{i+1}'}{c_{i+1}'^2}\\
%&=(1+\frac{p-i+1}{n-p}+O(\frac{p^2}{n^2}))\times(1+\frac{2(q-i-1)}{n}+O(\frac{q^2}{n^2}))\\
%&\quad\times \left(\frac{p-i}{n-p}+\frac{(p-i) (i+2)}{ n^2}+O(\frac{p^2q}{n^3})\right)\\
%&\quad\times\left(\frac{q-i}{n}+\frac{2(q-i)(q+1)}{n^2}+O(\frac{q^3}{n^3})\right)\\
%&=O(\frac{pq}{n^3})+\frac{(q-i)(p-i)}{n^2}(1+\frac{p+2q-3i}{n})\times\bigg(1+
%\frac{p+2q+i}{n}\bigg)\\
=\frac{(q-i)(p-i)}{n^2}(1+\frac{3p+3q-2i}{n})+O(\frac{pq}{n^3})
\endaligned $$ 
  and thereby
\begin{equation}\label{2cross21}\sum_{i=1}^{q-1}\mathbb{E} v_{i+1, i}^2 v_{i+1, i+1}^2=\frac{q(q-1)(3p-q-1)}{6n^2}+\frac{1}{n^3}(\frac{3}{2}p^2q(q-1)+\frac{2}{3}pq^3-\frac13q^4)+o(\frac{p^2q^2}{n^3}).\end{equation} 
Putting \eqref{4di}, \eqref{2cross2} and \eqref{2cross21} back into \eqref{sum00}, we know that 
$$
	\aligned \mathbb{E} \sum_{i=1}^n \frac{1}{(1-\nu_i)^2}
	&=q+\frac{2pq}{n-p}+\frac{q}{(n-p)^2}(p^2+6p-\frac52q-\frac13 q^2+\frac{17}6)+o\left(\frac{p^2 q^2}{n^3}\right)\\
	&\quad+\frac{q(q-1)}{n^2}(3p+\frac13q+\frac{17}{6})+\frac1{n^3}(3p^2q(3q-1)+\frac{14}{3}pq^3).\endaligned 
$$
The expression \eqref{cruciala} helps us again to get 
$$
	\aligned \frac{q}{(n-p)^2}(p^2+6p-\frac52q-\frac13 q^2+\frac{17}6)=\frac{p^2 q}{(n-p)^2}+\frac{6pq-\frac52q^2-\frac13 q^3+\frac{17}6q}{n^2}+\frac{12p^2 q-\frac23pq^3}{n^3}+o(\frac{p^2q^2}{n^3}).\endaligned 
$$
and then combining alike terms we see 
 $$
	\aligned \mathbb{E} \sum_{i=1}^n \frac{1}{(1-\nu_i)^2}
	&=q+\frac{2pq}{n-p}+\frac{p^2q}{(n-p)^2}+\frac{3pq(q+1)}{n^2}+\frac{9p^2q(q+1)+4pq^3}{n^3}+o(\frac{p^2q^2}{n^3}).\endaligned 
$$ 

The proof is completed now. 
\end{proof}

\subsection{The final statement} 

Having Lemma \ref{expressofI} and \ref{inversemj} at hands, and also it was proved in \cite{JM19} that 
$$\mathbb{E} \sum_{k=1}^q \lambda_k=\frac{pq}{n},$$ we are able to present the proof of Theorem \ref{main1}.  

In fact, Lemma \ref{expressofI} and Lemma \ref{inversemj} tell us  
$$\aligned  &\quad I(\mathcal{L}(\sqrt{n}{Z}_n)|\mathcal{L}({G}_n))\\
 &=\frac{n}4\sum_{k=1}^q\mathbb{E}\lambda_k-c_n(\sum_{k=1}^q \mathbb{E}(1-\lambda_k)^{-1}-q)+\frac{c_n^2}{n}\sum_{k=1}^q\mathbb{E}((1-\lambda_k)^{-2}-(1-\lambda_k)^{-1})\\
 &=\frac{pq}{4}-c_n(\frac{pq}{n-p}+\frac{pq(q+1)}{n^2}+\frac{2p^2q(q+1)+pq^3}{n^3})+o(\frac{p^2q^2}{n^2})\\
 &\quad+\frac{c_n^2}{n}(\frac{pq}{n-p}+\frac{p^2q}{(n-p)^2}+\frac{2pq(q+1)}{n^2}+\frac{7p^2q(q+1)+3pq^3}{n^3}).
\endaligned $$
Remembering $c_n=\frac{n-p-q-1}{2},$  we see clearly that 
$$\aligned & \quad c_n\left(\frac{p q}{n-p} + \frac{p q(q+1)}{n^2}+\frac{2p^2q(q+1)+pq^3}{n^3}\right)\\
&=\frac{pq}{2}-\frac{pq(q+1)}{2(n-p)}+\frac{pq(q+1)}{2n}-\frac{pq(q+1)(p+q+1)}{2n^2}+\frac{2p^2q(q+1)+pq^3}{2n^2}\\
&=\frac{pq}{2}+o(\frac{p^2q^2}{n^2})\endaligned $$
and 
$$\aligned &\quad \frac{c_n^2}{n}\left(\frac{p q}{n-p}+\frac{2pq(q+1)}{n^2} + \frac{p^2q}{(n-p)^2}\right)\\
&=\frac{pq}{4}(1-\frac{p+q+1}{n})(1-\frac{q+1}{n-p})+\frac{pq(q+1)}{2n}(1-\frac{p+q+1}{n})^2+\frac{p^2q}{4n}(1-\frac{q+1}{n-p})^2\\
&=\frac{pq}{4}-\frac{pq(p+2(q+1))}{4n}+\frac{pq(q+1)}{2n}+\frac{p^2q}{4n}+\frac{pq^3}{4n^2}+o(\frac{p^2q^2}{n^2})\\
&\quad-\frac{pq(q+1)(p+q+1)}{n^2}-\frac{p^2q(q+1)}{2n^2}\\
&=\frac{pq}{4}-\frac{6p^2q(q+1)+3pq^3}{4n^2}+o(\frac{p^2q^2}{n^2}).\endaligned$$

Thereby, it holds 
$$\aligned   \quad I(\mathcal{L}(\sqrt{n}{Z}_n)|\mathcal{L}({G}_n))&=-\frac{6p^2q(q+1)+3pq^3}{4n^2}+\frac{7p^2q(q+1)+3pq^3}{4n^2}+o(\frac{p^2q^2}{n^2})\\
&=\frac{p^2q(q+1)}{4n^2}+o(\frac{p^2q^2}{n^2}).
\endaligned $$

The proof of Theorem \ref{main1} is finally completed.

\begin{rmk} 
Let $(\lambda_k)_{1\le k\le q}$ and $(\nu_k)_{1\le k\le q}$ be the eigenvalues of $Z_n  Z_n'$ and $G_n  G_n',$ respectively. Set $\boldsymbol{\lambda}=(\lambda_1, \cdots, \lambda_q)$ and $\boldsymbol{\nu }=(\nu_1, \cdots, \nu_q).$ 
Since $\boldsymbol{\lambda}$ is identified as a Jacobi ensemble, $\boldsymbol{\nu}$ could  also be regarded as a Laguerre ensemble as well. Thus, according to the Fisher information between Laguerre and Jacobi ensembles offered in \cite{Ma25}, we see  
$$I(\mathcal{L}(\boldsymbol{\lambda})|\mathcal{L}(\boldsymbol{\nu}))=\frac{pq^2+q^3}{4n^2}+o(\frac{pq^2}{n^2})$$ when $p=o(n).$ The Fisher information among the distributions of corresponding eigenvalues  $\boldsymbol{\lambda}$ and $\boldsymbol{\nu}$ behaves in great difference from that for joint distributions of entries of matrices given in Theorem \ref{main1}. 
\end{rmk}

\begin{rmk} 
We can have fluctuations and large deviations  for the empirical of $(\lambda_k)_{1\le k\le q}$ as well as those for the extremal statistics $\lambda_{(q)}:=\max_{1\le i\le q}\lambda_i$ and $\lambda_{(1)}:=\min_{1\le i\le q}\lambda_i$ by the limit theorems obtained for general $\beta$-ensembles (see \cite{LM23}, \cite{Ma23}, \cite{Ma25} and \cite{MW25}). To avoid losing sight of the main point,  we only state the stong law of large number and the fluctuations of extremal statistics.  Suppose  
$$({\bf A}):  \quad \quad \quad \quad \quad \quad \quad \quad q\gg 1 \quad \quad \lim_{n\to\infty}\frac{q}{p}=\gamma\quad \quad \text{and} \quad \quad \lim_{n\to\infty}\frac{pq}{n}=0. $$ 
On the one hand, one gets the strong law of large number as 
	$$\frac{n\lambda_{(q)}-p}{\sqrt{pq}}\longrightarrow 2+\sqrt{\gamma}  \quad \text{and} \quad \frac{n\lambda_{(1)}-p}{\sqrt{pq}}\longrightarrow 2-\sqrt{\gamma}$$ almost surely as $n$ tends to the infinity. 
On the other hand, we gets the fluctuations as follows.
\begin{itemize} 
\item[(1).] If $\gamma\in [0, 1],$ then 
$$\frac{(pq)^{1/6}}{(\sqrt{p}+\sqrt{q})^{4/3}}\left(n\,\lambda_{(q)}-(\sqrt{p}+\sqrt{q})^2\right) $$
converges weakly to $F_1$ 
as $n\to\infty.$ 
\item[(2).] If $\gamma\in [0, 1),$ then $$\frac{(pq)^{1/6}}{(\sqrt{p}-\sqrt{q})^{4/3}}\left((\sqrt{p}-\sqrt{q})^2-n\,\lambda_{(1)}\right) $$
converges weakly to $F_1$ 
as $n\to\infty.$ 
\item[(3).]  If $p=q,$ then 
$\frac1q \sqrt{\frac{\lambda_{(q)}}{\lambda_{(1)}}}$ converges weakly to the probability with density function 
$h(x)=8x^{-3} e^{-4/x^2}$ for $x>0.$ 
\end{itemize}

Here, $F_1$ is the Tracy-Widom law, originally appeared as the weak limit for extremal statistics for Gaussian orthogonal ensemble (see \cite{TW1}),  whose distribution function is given by 
$$
F_{1}(x)=\exp \left(-\frac{1}{2} \int_{x}^{\infty} q(y) d y\right)\exp \left(-\frac{1}{2}\int_{x}^{\infty}(y-x) q^{2}(y) d y\right)
$$
for all $x \in \mathbb{R}$ 
with  $q$ is the unique solution to the Painlev\'e II differential equation
$$
q^{\prime \prime}(x)=x q(x)+2 q^{3}(x)
$$
satisfying  the boundary condition $q(x) \sim A i(x)$ as $x \rightarrow \infty$ and $A i(x)$ is the Airy function.
\end{rmk}

%\noindent\textbf{Acknowledgement}. 

\end{document}